\documentclass[11pt,a4paper,twoside]{article}

\usepackage[main=english]{babel}
\usepackage[latin1]{inputenc}
\usepackage[T1]{fontenc}
\usepackage{microtype}

\usepackage[margin=3cm]{geometry}
\usepackage{amssymb}
\usepackage{mathtools}
\usepackage{commath}
\usepackage{amsthm}
\usepackage{eucal}
\usepackage{eufrak}
\usepackage{xcolor}
\usepackage[bookmarks,colorlinks]{hyperref}
\hypersetup{colorlinks,citecolor={blue}}

\numberwithin{equation}{section}

\title{\textbf{Nearly K\"ahler six-manifolds with two-torus symmetry}}
\date{}
\author{Giovanni Russo and Andrew Swann}

\theoremstyle{plain}
\newtheorem{prop}{Proposition}[section]
\newtheorem{theorem}[prop]{Theorem}
\newtheorem{lemma}[prop]{Lemma}

\theoremstyle{remark}
\newtheorem{remark}[prop]{Remark}

\DeclareMathOperator{\id}{Id}
\DeclareMathOperator{\vol}{\mathrm{vol}}
\DeclareMathOperator{\chair}{\lrcorner \thinspace}
\DeclareMathOperator{\Span}{Span}

\newcommand{\Lie}{\mathcal{L}}

\begin{document}

\maketitle

\begin{abstract}
We consider nearly K\"ahler 6-manifolds with effective 2-torus symmetry. The multi-moment map for the $T^2$-action becomes an eigenfunction of the Laplace operator. At regular values, we prove the $T^2$-action is necessarily free on the level sets and determines the geometry of three-dimensional quotients. An inverse construction is given locally producing nearly K\"ahler six-manifolds from three-dimensional data. This is illustrated for structures on the Heisenberg group.
\end{abstract}

\tableofcontents

\section{Introduction}
\label{intro}

Nearly K\"ahler manifolds were originally introduced by Gray \cite{Gray1} as follows: let $M$ be an almost Hermitian manifold with Riemannian metric $g$, almost complex structure $J$, and Levi-Civita connection $\nabla$, then $M$ is called nearly K\"ahler if $(\nabla_X J)X = 0$ for each vector field $X$ on $M$. As Nagy showed \cite{nagy_six} the important case is that of compact, six-dimensional nearly K\"ahler manifolds that are \emph{not} K\"ahler, and this is what we focus on. In \cite{Reyes1993}, Carri\'on discusses how the definition above in dimension 6 is related to a system of partial differential equations. Let us define the fundamental two-form $\sigma \coloneqq g(J{}\cdot{},{}\cdot{})$. It turns out that Gray's definition is equivalent to the existence of a complex three-form $\psi_{\mathbb{C}} = \psi_+ + i\psi_-$ such that
\begin{equation}
\label{nk_condition}
\begin{cases}
\dif \sigma = 3\psi_+ \\
\dif \psi_- = -2\sigma \wedge \sigma.
\end{cases}
\end{equation}
B\"ar \cite{Bar} describes the interplay between nearly K\"ahler 6-manifolds and $\mathrm{G}_2$-geometry: let $C= (M \times \mathbb{R}_{>0}, g_C \coloneqq t^2g+\dif t^{\otimes 2})$ be a seven-dimensional Riemannian cone over a smooth compact manifold $(M,g)$. Assume that the holonomy of the cone is contained in the exceptional compact Lie group $\mathrm{G}_2$. Then there exists a pair of closed differential forms $\varphi$ and ${*}\varphi$, pointwise equivalent to the model forms on $\mathbb{R}^7$. Viewing $M$ as the level set $t=1$ in $C$, the restriction and contraction by $\partial/\partial t$ of $\varphi$ and ${*}\varphi$ define an $\mathrm{SU}(3)$-structure $(\sigma,\psi_{\mathbb{C}}=\psi_++i\psi_-)$ on $M$ such that
\begin{equation*}
\varphi = t^2\dif t \wedge \sigma + t^3 \psi_+, \quad {*}\varphi = \tfrac{1}{2}t^4\sigma \wedge \sigma + t^3\psi_- \wedge \dif t.
\end{equation*}
The closedness of $\varphi$ and ${*}\varphi$ is equivalent to $\dif \sigma = 3\psi_+, \dif \psi_- = -2\sigma \wedge \sigma$, so $M$ is nearly K\"ahler. Since every manifold with holonomy contained in $\mathrm{G}_2$ is Ricci-flat, nearly K\"ahler 6-manifolds are necessarily Einstein with positive scalar curvature. The latter result was actually proved by Gray (\cite{Gray1976}, Theorem 5.2) before the connection with $\mathrm{G}_2$-holonomy had been noticed. In particular, a complete connected nearly K\"ahler 6-manifold is compact with finite fundamental group, and its universal cover is also a complete nearly K\"ahler manifold. This is why we can restrict to the connected, simply connected case.

There are a few known examples of nearly K\"ahler 6-manifolds: Butruille \cite{Butruille} proved that the only homogeneous ones are $\mathbb{S}^6, \mathbb{CP}^3, F_{1,2}(\mathbb{C}^3)$, and $\mathbb{S}^3 \times \mathbb{S}^3$. Recently, Foscolo and Haskins \cite{Foscolo} showed the existence of the first complete inhomogeneous examples by proving there is at least a cohomogeneity one nearly K\"ahler structure on $\mathbb{S}^6$ and $\mathbb{S}^3 \times \mathbb{S}^3$.

The aim of the present work is to lay the foundations for a path to find new examples with weaker symmetry: our starting remark is that the homogeneous cases listed in \cite{Butruille} and the examples in \cite{Foscolo} always contain a 2-torus $T^2$ as group of symmetries. So we start introducing nearly K\"ahler 6-manifolds with 2-torus symmetry and construct a multi-moment map for the torus action. A remarkable result is that the multi-moment map is an eigenfunction of the Laplace operator (cf.\ \cite{Nagy}). After discussing some consequences of the latter fact, we use this map to construct a 3-manifold via $T^2$-reduction. Then we describe the inverse process: we get the necessary conditions to construct a principal $T^2$-bundle over a 3-manifold, and find the evolution equations that must be satisfied by the geometric structure on its total space in order to get a six-dimensional nearly K\"ahler manifold. To this purpose, we mimic the construction developed by Madsen and Swann in \cite{MadsenSwann} on the reduction of torsion-free $\mathrm{G}_2$-manifolds.

\medbreak
\textbf{Acknowledgements}. We thank Paul-Andi Nagy and Uwe Semmelmann for useful discussions. Both authors are supported by the Danish Council for Independent Research | Natural Sciences Project DFF - 6108-00358, and by the Danish National Research Foundation grant DNRF95 (Centre for Quantum Geometry of Moduli Spaces).

\section{Torus symmetry and multi-moment maps}
\label{second_section}

Let $M$ be a closed, connected, six-dimensional smooth manifold, and equip it with an $\mathrm{SU}(3)$-structure given by a Riemannian metric $g$, an almost complex structure $J$, isometry of each tangent space, and a complex volume form $\psi_{\mathbb{C}} = \psi_++i\psi_-$ of type $(3,0)$. Consider the two-form $\sigma \coloneqq g(J{}\cdot{},{}\cdot{})$. We say that $(M,g,J,\psi_{\mathbb{C}})$ is \emph{nearly K\"ahler} if the equations \eqref{nk_condition} are satisfied. From now on we assume that our structure has an effective $T^2$-symmetry, so that a 2-torus acts effectively on $M$ preserving $g,J$ and $\psi_{\pm}$. Denote by $U$ and $V$ the generating vector fields for the action. Since $T^2$ is an Abelian group, we get $\sbr{U,V}=0$.

We introduce a multi-moment map for the torus action, adapting the construction of Madsen and Swann \cite{MadsenSwann} to the present case. Let us define the \emph{Lie kernel} $P_{\mathfrak{t}^2}$ as the $\mathfrak{t}^2$-submodule of $\Lambda^2 \mathfrak{t}^2$ given by $P_{\mathfrak{t}^2} \coloneqq \ker \left(L \colon \Lambda^2 \mathfrak{t}^2 \rightarrow \mathfrak{t}^2 \right)$, where $L$ is the $\mathfrak{t}^2$-linear map induced by the Lie bracket. Since $T^2$ is Abelian, $P_{\mathfrak{t}^2} = \Lambda^2 \mathfrak{t}^2 \cong \mathbb{R}$. Observe that by the first equation in \eqref{nk_condition} the three-form $\psi_+$ is closed, so the pair $(M,3\psi_+)$ is a \emph{strong geometry} (cf.\ \cite{MadsenSwann}). Now consider the map
\begin{equation*}
\nu_M \colon M \rightarrow P_{\mathfrak{t}^2}^* \cong \mathbb{R}, \quad \nu_M(p) \coloneqq \sigma_p(U_p,V_p).
\end{equation*}
Recall that the $T^2$-action preserves $g$ and $J$, so it preserves $\sigma$, therefore the Lie derivatives $\Lie_U \sigma$ and $\Lie_V \sigma$ vanish. Further, since $U$ and $V$ commute we have $\Lie_V(U \chair \sigma) = U \chair \Lie_V \sigma = 0$. The action preserves $U$ and $V$ as well, so the function $\nu_M$ is invariant under the $T^2$-action. We can then compute $\dif \nu_M$ using Cartan's formula:
\begin{align*}
\dif \nu_M = \dif \left(V \chair U \chair \sigma\right) & = \Lie_V (U \chair \sigma) - V \chair \dif \left(U \chair \sigma\right) \\
& = -V \chair \Lie_U \sigma + V \chair U \chair \dif \sigma = 3\psi_+(U,V,{}\cdot{}).
\end{align*}
This shows that $\nu_M$ is a multi-moment map for the $T^2$-action on $M$.

In what follows we will denote $g(U,U),g(U,V),g(V,V)$ by $g_{UU},g_{UV},g_{VV}$ and define a real function $h$ on $M$ satisfying
\begin{equation}
\label{h_property}
h^2 = g_{UU}g_{VV}-g_{UV}^2.
\end{equation}
When there is no danger of confusion we still use the same notation pointwise. Lastly, when we work on the complexified tangent bundle of $M$, we extend $g,J$ and $\psi_{\pm}$ by $\mathbb{C}$-linearity in all their arguments.

Let $\{F_k,JF_k\}_{k=1,2,3}$ be an $\mathrm{SU}(3)$-basis of $T_pM$ and $\{f_k,Jf_k\}_{k=1,2,3}$ be its dual. Then, at $p$,
\begin{align}
\label{sigma_pointwise}
\sigma & = f_1 \wedge Jf_1 + f_2 \wedge Jf_2 + f_3 \wedge Jf_3, \\
\psi_{\mathbb{C}} & = \bigl(f_1+iJf_1\bigr) \wedge \bigl(f_2+iJf_2\bigr) \wedge \bigl(f_3+iJf_3\bigr).
\end{align}
The volume form splits into real and imaginary parts:
\begin{align}
\label{real_volume}
\psi_+ & = f_1 \wedge f_2 \wedge f_3 - Jf_1 \wedge Jf_2 \wedge f_3 - f_1 \wedge Jf_2 \wedge Jf_3 - Jf_1 \wedge f_2 \wedge Jf_3, \\
\label{imaginary_volume}
\psi_- & = f_1 \wedge f_2 \wedge Jf_3 - Jf_1 \wedge Jf_2 \wedge Jf_3 + f_1 \wedge Jf_2 \wedge f_3 + Jf_1 \wedge f_2 \wedge f_3.
\end{align}
Using the action of $\mathrm{SU}(3)$, we can assume without loss of generality that $U_p = x_0F_1$, for some real number $x_0$, and $V_p = x_1F_1+y_1JF_1+x_2F_2$, for some real numbers $x_1,y_1,x_2$. Since $x_1^2+y_1^2+x_2^2 = g_{VV}, \nu_M(p) = g(JU_p,V_p)$ and $g_{UV}=g_{UU}^{1/2}x_1$, we get
\begin{equation*}
U_p = g_{UU}^{1/2}F_1, \quad g_{UU}^{1/2}V_p = \Bigl(g_{UV}F_1+\nu_MJF_1+\bigl(h^2-\nu_M^2\bigr)^{1/2}F_2\Bigr).
\end{equation*}
Thus, in general $\dif \nu_M = 3\psi_+(U,V,{}\cdot{}) = 3\bigl(h^2-\nu_M^2\bigr)^{1/2}f_3$ pointwise.
Lastly, observe that
\begin{equation}
\label{useful_formula}
V \chair U \chair \left(\sigma \wedge \sigma\right) = 2\bigl(\nu_M \sigma - (U \chair \sigma) \wedge (V \chair \sigma)\bigr).
\end{equation}
This formula is obtained by expanding the left-hand side.
\section{Properties of the multi-moment map and regular values}
Using the set-up we can then prove the following result:
\begin{lemma}
\label{laplace_prop}
Let $\Delta$ be the Laplace operator on $C^{\infty}(M)$, defined as $\Delta = \dif^{\thinspace *}\negthinspace\dif \coloneqq -{*}\negthinspace\dif{*}\negthinspace\dif{}$. Then
\begin{equation}
\label{laplace}
\Delta \nu_M = 24 \nu_M.
\end{equation}
\end{lemma}
\begin{proof}
Firstly, we show that
\begin{equation}
\label{Hodge1}
{*}\negthinspace\dif \nu_M = {*}\bigl(3\psi_+(U,V,{}\cdot{})\bigr) = \frac{3}{2}\sigma \wedge \sigma \wedge \alpha_0,
\end{equation}
where $\alpha_0 \coloneqq V \chair U \chair \psi_-$. From equations \eqref{sigma_pointwise} and \eqref{imaginary_volume} we get pointwise
\begin{align*}
\sigma \wedge \sigma & = 2(f_1 \wedge Jf_1 \wedge f_2 \wedge Jf_2 + f_1 \wedge Jf_1 \wedge f_3 \wedge Jf_3 + f_2 \wedge Jf_2 \wedge f_3 \wedge Jf_3), \\
\alpha_0 & = \bigl(h^2 - \nu_M^2\bigr)^{1/2}Jf_3.
\end{align*}
Hence $\sigma \wedge \sigma \wedge \alpha_0 = 2\bigl(h^2 - \nu_M^2\bigr)^{1/2}f_1 \wedge Jf_1 \wedge f_2 \wedge Jf_2 \wedge Jf_3$, and
$$3\psi_+(U,V,{}\cdot{}) \wedge (\sigma \wedge \sigma \wedge \alpha_0) = \frac{2}{3}\bigl\lVert 3\psi_+(U,V,{}\cdot{}) \bigr\rVert_g^2 \vol_M$$
from which we get \eqref{Hodge1}. Thus $\dif{*}\negthinspace\dif\nu_M = \frac{3}{2}\sigma \wedge \sigma \wedge \dif\alpha_0$, since $\dif \left(\sigma \wedge \sigma\right) = 0$. By Cartan's formula $\dif\alpha_0 = \dif \left(V \chair U \chair \psi_-\right) = -2V \chair U \chair (\sigma \wedge \sigma)$. Then, formula \eqref{useful_formula} yields
\begin{align*}
\dif{*}\negthinspace\dif \nu_M & = -6\nu_M \sigma \wedge \sigma \wedge \sigma + 6\sigma \wedge \sigma \wedge (U \chair \sigma) \wedge (V \chair \sigma) \\
& = -36\nu_M \vol_M + 12 \nu_M \vol_M \\
& = -24\nu_M \vol_M,
\end{align*}
and we are done.
\end{proof}
A consequence of this result is
\begin{prop}
\label{average}
The average value of the multi-moment map $\nu_M$ is $0$. Moreover, the range of $\nu_M$ is a compact interval containing $0$ in its interior.
\end{prop}
\begin{proof}
Since $M$ has no boundary we can apply Stokes' theorem and Lemma \ref{laplace_prop}:
\begin{align*}
24 \int_M  \nu_M \vol_M = &\int_M \Delta \nu_M \vol_M = \int_M \dif^{\thinspace *}\negthinspace\dif\nu_M \wedge {*}1 = -\int_M {*}\negthinspace\dif{*}\negthinspace\dif \nu_M \wedge {*}1\\
= & \int_M \dif \thinspace ({*}\negthinspace\dif\nu_M) = \int_{\partial M} {*}\negthinspace \dif \nu_M = 0.
\end{align*}
So we have our first claim $$\int_M \nu_M \vol_M = 0.$$
However, $\nu_M$ is not constantly zero, and we prove this by contradiction. Assume that $M = \nu_M^{-1}(0)$. Then $\sigma(U,V) \equiv 0$, so $V = \eta U + W$ for some real function $\eta$ and some $W$ orthogonal to $U$ and $JU$. Using the action of $\mathrm{SU}(3)$ we can assume pointwise $U=x_0F_1$, and $W = y_0F_2$ for some real numbers $x_0,y_0$. Then $0 = \dif \nu_M = 3\psi_+(U,V,{}\cdot{}) = 3x_0y_0f_3$. If $U \neq 0$, this implies $y=0$, namely $V = \eta U$. Therefore $V$ and $U$ are linearly dependent pointwise, hence the action of $T^2$ is not effective on $M$ (\cite{KobNom}, Proposition 4.1), which is a contradiction. This allows us to write $\nu_M\colon M \rightarrow \sbr{a,b}$, where $a < 0 < b$, for $M$ is compact and connected, and $\nu_M$ is smooth.
\end{proof}

Now assume $s$ is a regular value for $\nu_M$. In the next proposition we show that the $T^2$-action on $\nu_M^{-1}(s)$ is \emph{free}, so $\nu_M^{-1}(s)/T^2$ is a smooth $3$-manifold and $\nu_M^{-1}(s) \rightarrow \nu_M^{-1}(s)/T^2$ is a principal $T^2$-bundle. Its base space is the \emph{$T^2$-reduction of $M$ at level $s$}. We will study the geometry of the quotients $\nu_M^{-1}(s)/T^2$ in Section \ref{three_dimensional_quotients}.

\begin{prop}
\label{action_free}
The multi-moment map $\nu_M$ has non-zero regular values. For any regular value $s$ the $T^2$-action on $\nu_M^{-1}(s)$ is free. Thus $\nu_M^{-1}(s)/T^2$ are smooth three-dimensional manifolds.
\end{prop}
\begin{proof}
By Sard's theorem the set of critical values has Lebesgue measure $0$ in $\sbr{a,b}$, so we can assert there exist infinitely many regular values for $\nu_M$ in $(a,b) \neq \varnothing$. Let $s$ be any of them. Then $(\dif \nu_M)_p$ has rank $1$ for each $p \in \nu_M^{-1}(s)$, so $3\psi_+(U,V,{}\cdot{})_p = (\dif \nu_M)_p \neq 0$. Thus $U,V$ are linearly independent over $\mathbb{C}$ on $\nu_M^{-1}(s)$, and this yields a locally free action on $\nu_M^{-1}(s)$.

On the other hand, if $H$ is the stabilizer in $T^2$ of some $p \in \nu_M^{-1}(s)$, it preserves $g,J$ and $\psi_{\pm}$ by hypothesis, and also $U$ and $V$. Hence, $H$ fixes $U,JU,V,JV,W \coloneqq \psi_+(U,V,{}\cdot{})^{\sharp}$ and $JW$, so all of $T_pM$. Then, by the Tubular Neighbourhood Theorem, the set $A \coloneqq \{p \in M : h\cdot p = p, \dif h_p = \id_{T_pM} \text{ for each } h \in H\}$ is open and closed in $M$. Since we assume $M$ connected and $A$ is not empty, $A = M$. If $H$ is not trivial, then we get a contradiction, because the action is effective on $M$. So $H$ must be trivial, and our claim is proved.
\end{proof}

\section{Reduction to three-manifolds}
\label{three_dimensional_quotients}

In order to study the geometric structure of the quotients $Q_s^3 \coloneqq \nu_M^{-1}(s)/T^2$, with $s \neq 0$ a regular value for $\nu_M$, we need to determine which forms on $\nu_M^{-1}(s)$ descend to them. We use the following criterion: a $k$-form $\beta$ on $\nu_M^{-1}(s)$ descends to $Q_s^3$ if and only if it is basic, that is $\Lie_U \beta = \Lie_V \beta = 0$ and $U \chair \beta = V \chair \beta = 0$. Before starting, let us define the dual forms of $U$ and $V$:
\begin{equation*}
\vartheta_1 \coloneqq h^{-2}\bigl(g_{VV}U^{\flat}-g_{UV}V^{\flat}\bigr), \quad \vartheta_2 \coloneqq h^{-2}\bigl(g_{UU}V^{\flat}-g_{UV}U^{\flat}\bigr),
\end{equation*}
where $h$ satisfies \eqref{h_property}. The pair $(\vartheta_1,\vartheta_2)$ is a connection one-form for the $T^2$-bundle $\nu_M^{-1}(s) \rightarrow \nu_M^{-1}(s)/T^2$. The invariant functions we find are $g_{UU},g_{UV},g_{VV}$. Then we have basic one-forms
\begin{equation*}
\alpha_0 \coloneqq V \chair U \chair \psi_-, \quad \alpha_1 \coloneqq s\vartheta_1 + V \chair \sigma, \quad \alpha_2 \coloneqq s\vartheta_2 - U \chair \sigma,
\end{equation*}
and lastly the two-forms $U \chair \psi_+, V \chair \psi_+$ are basic.

The next step is to specify $g,\sigma,\psi_{\pm}$ on $M$ in terms of the forms $\dif \nu_M, \vartheta_1, \vartheta_2, \alpha_0, \alpha_1, \alpha_2$. We work on $M$, pointing out what holds in particular on the level sets $\nu_M^{-1}(s)$. Using the basis of the previous section, we find pointwise:
\begin{align*}
\dif \nu_M & = 3\bigl(h^2-\nu_M^2\bigr)^{1/2} f_3, \\[5pt]
\vartheta_1 & = h^{-2}g_{UU}^{-1/2}\Bigl(h^2f_1 - g_{UV}\nu_MJf_1 - g_{UV}\bigl(h^2-\nu_M^2\bigr)^{1/2}f_2\Bigr),\\[5pt]
\vartheta_2 &= h^{-2}g_{UU}^{1/2} \Bigl(\nu_MJf_1 + \bigl(h^2-\nu_M^2\bigr)^{1/2}f_2\Bigr), \\[5pt]
\alpha_0 &= \bigl(h^2-\nu_M^2\bigr)^{1/2} Jf_3,\\[5pt]
\alpha_1 &= h^{-2}g_{UU}^{-1/2} \bigl(h^2-\nu_M^2\bigr)^{1/2} \Bigl(g_{UV}\bigl(h^2-\nu_M^2\bigr)^{1/2}Jf_1 - g_{UV}\nu_Mf_2 + h^2Jf_2\Bigr),\\[5pt]
\alpha_2 &= h^{-2}g_{UU}^{1/2} \bigl(h^2-\nu_M^2\bigr)^{1/2} \Bigl(\nu_Mf_2-\bigl(h^2-\nu_M^2\bigr)^{1/2}Jf_1\Bigr).
\end{align*}
Now let $s \neq 0$. The last five forms are linearly independent on $\nu_M^{-1}(s)$ if and only if $h^2 \neq s^2$. Observe that $h^2 - \nu_M^2 = \lVert (h^2-\nu_M^2)^{1/2}f_3 \rVert_g^2 =\tfrac{1}{9} \lVert \dif \nu_M \rVert_g^2 = \lVert \psi_+(U,V,{}\cdot{}) \rVert_g^2$, and this quantity is non-zero under our assumptions. Therefore, we can compute the general expressions of $g, \sigma, \psi_{\pm}$ on $M$:
\begin{align}
\label{metric} g & = \frac{1}{9\bigl(h^2 - \nu_M^2\bigr)}\dif \nu_M^{\otimes 2}+ g_{UU}\vartheta_1^{\otimes 2} + g_{VV}\vartheta_2^{\otimes 2} + g_{UV}\bigl(\vartheta_1 \otimes \vartheta_2+\vartheta_2 \otimes \vartheta_1\bigr) \nonumber \\
& \qquad +\frac{1}{h^2 - \nu_M^2}\Bigl( \alpha_0^{\otimes 2}+ g_{UU}\alpha_1^{\otimes 2} + g_{VV}\alpha_2^{\otimes 2}+g_{UV}\bigl(\alpha_1 \otimes \alpha_2 + \alpha_2 \otimes \alpha_1\bigr)\Bigr); \\[10pt]
\label{sigma} \sigma & = \frac{1}{3\bigl(h^2 - \nu_M^2\bigr)}\dif \nu_M \wedge \alpha_0+ \nu_M\vartheta_1 \wedge \vartheta_2 - \vartheta_1 \wedge \alpha_2 + \vartheta_2 \wedge \alpha_1 - \frac{\nu_M}{h^2 - \nu_M^2}\alpha_1 \wedge \alpha_2;\\[10pt]
\label{psi_+} \psi_+ & = \frac{1}{3\bigl(h^2-\nu_M^2\bigr)}\dif \nu_M \wedge \Bigl(\bigl(h^2-\nu_M^2\bigr)\vartheta_1 \wedge \vartheta_2+\nu_M \bigl(\vartheta_1 \wedge \alpha_2 - \vartheta_2 \wedge \alpha_1\bigr) - \alpha_1 \wedge \alpha_2 \Bigr) \nonumber \\
& \qquad - \frac{1}{h^2 - \nu_M^2}\Bigl(\vartheta_1 \wedge \bigl(g_{UU} \alpha_1 + g_{UV}\alpha_2\bigr) + \vartheta_2 \wedge \bigl(g_{UV} \alpha_1 + g_{VV} \alpha_2\bigr)\Bigr) \wedge \alpha_0; \\[10pt]
\label{psi_-}\psi_- & = \frac{1}{3\bigl(h^2 - \nu_M^2\bigr)} \dif \nu_M\wedge \Bigl(\vartheta_1 \wedge \bigl(g_{UU}\alpha_1 + g_{UV} \alpha_2\bigr) + \vartheta_2 \wedge \bigl(g_{UV}\alpha_1+ g_{VV} \alpha_2\bigr) \Bigr) \nonumber \\
& \qquad + \frac{1}{h^2 - \nu_M^2}\Bigl(\bigl(h^2-\nu_M^2\bigl)\vartheta_1 \wedge \vartheta_2 + \nu_M\bigl(\vartheta_1 \wedge \alpha_2 - \vartheta_2 \wedge \alpha_1\bigr) - \alpha_1 \wedge \alpha_2\Bigr)\wedge \alpha_0.
\end{align}

If we use the nearly K\"ahler structure equations \eqref{nk_condition} we get further relationships. The cotangent space of $M$ splits as the direct sum $V \oplus H$, where $\vartheta_i \in V, i = 1,2$ and $H$ contains $\dif \nu_M, \alpha_k, k = 0,1,2$. Comparing coefficients in $\dif\sigma = 3\psi_+$ we obtain
\begin{gather}
\label{theta_2*}
\nu_M \dif \vartheta_2 = \dif \alpha_2 + \frac{1}{h^2 - \nu_M^2}\bigl(3g_{UU}\alpha_1 \wedge \alpha_0 + 3g_{UV}\alpha_2 \wedge \alpha_0+\nu_M \dif \nu_M \wedge \alpha_2 \bigr), \\[10pt]
\label{theta_1*}
\nu_M \dif \vartheta_1 = \dif \alpha_1 - \frac{1}{h^2 -\nu_M^2}\bigl(3g_{UV}\alpha_1 \wedge \alpha_0 + 3g_{VV}\alpha_2 \wedge \alpha_0 - \nu_M \dif \nu_M \wedge \alpha_1 \bigr), \\[10pt]
\label{forme1*}
\begin{split}
\dif \vartheta_1 \wedge \alpha_2 - \dif \vartheta_2 \wedge \alpha_1 & = \Biggl( \frac{2\nu_M^2}{\bigl(h^2 - \nu_M^2\bigr)^2}\dif \nu_M  - \frac{\nu_M}{\bigl(h^2 - \nu_M^2\bigr)^2}\dif\thinspace (h^2) \Biggr) \wedge \alpha_2 \wedge \alpha_1 \\
& \qquad + \dif \nu_M \wedge \Biggl(\frac{1}{3\bigl(h^2 - \nu_M^2\bigr)^2}\dif\thinspace (h^2) \wedge \alpha_0 - \frac{1}{3\bigl(h^2 - \nu_M^2\bigr)}\dif \alpha_0 \Biggr) \\
& \qquad + \frac{\nu_M}{h^2 - \nu_M^2}\dif \left(\alpha_2 \wedge \alpha_1\right).
\end{split}
\end{gather}

The equation $\dif \psi_- = -2\sigma \wedge \sigma$ gives
\begin{equation}
\label{alpha_0}
\dif \alpha_0 = - \frac{4\nu_M}{3\bigl(h^2 - \nu_M^2\bigr)}\dif \nu_M \wedge \alpha_0 +\frac{4h^2}{h^2 - \nu_M^2}\alpha_1 \wedge \alpha_2,
\end{equation}
\begin{align}
\label{forme2*}
\dif \vartheta_2 \wedge \alpha_0 & = \frac{-1}{3\bigl(h^2 - \nu_M^2\bigr)^2} \Bigl(\dif \nu_M \wedge \bigl(g_{UU} \alpha_1 + g_{UV} \alpha_2\bigr) +3\nu_M\alpha_0 \wedge \alpha_2 \Bigr)\wedge \dif \thinspace (h^2) \nonumber \\
& \qquad - \frac{1}{3\bigl(h^2 - \nu_M^2\bigr)}\Bigl(\dif \nu_M \wedge \dif \,\bigl(g_{UU}\alpha_1 + g_{UV} \alpha_2\bigr)+3\nu_M \dif \alpha_2 \wedge \alpha_0\Bigr) \nonumber \\
& \qquad - \frac{h^2 - 3\nu_M^2}{3\bigl(h^2 - \nu_M^2\bigr)^2}\dif \nu_M \wedge \alpha_0 \wedge \alpha_2, \\
\label{forme3*}
\dif \vartheta_1 \wedge \alpha_0 & = \frac{1}{3\bigl(h^2 - \nu_M^2\bigr)^2} \Bigl(\dif \nu_M \wedge \bigl(g_{UV} \alpha_1 + g_{VV} \alpha_2\bigr) -3\nu_M \alpha_0 \wedge \alpha_1 \Bigr) \wedge \dif \thinspace (h^2) \nonumber \\
& \qquad + \frac{1}{3\bigl(h^2 - \nu_M^2\bigr)}\Bigl(\dif \nu_M \wedge \dif \,\bigl(g_{UV} \alpha_1+ g_{VV} \alpha_2\bigr)-3\nu_M \dif \alpha_1 \wedge \alpha_0\Bigr) \nonumber \\
& \qquad - \frac{h^2 -3\nu_M^2}{3\bigl(h^2 - \nu_M^2\bigr)^2}\dif \nu_M \wedge \alpha_0 \wedge \alpha_1.
\end{align}
The relations among $\alpha_0,\alpha_1,\alpha_2$ on the $T^2$-reduction at level $s$ are then
\begin{align}
\label{alpha_0_ls}
\dif \alpha_0 & = \frac{4h^2}{h^2 - s^2}\alpha_1 \wedge \alpha_2, \\
\label{alpha_1_ls}
\dif \alpha_1 \wedge \alpha_0 & = \frac{s^2}{h^2 \bigl(h^2 - s^2\bigr)}\dif\thinspace (h^2) \wedge \alpha_1 \wedge \alpha_0, \\
\label{alpha_2_ls}
\dif \alpha_2 \wedge \alpha_0 & = \frac{s^2}{h^2 \bigl(h^2 - s^2\bigr)}\dif\thinspace (h^2) \wedge \alpha_2 \wedge \alpha_0.
\end{align}
Define $f\coloneqq 4h^2/(h^2-s^2)$. Observe that $f>4$, and
\begin{equation*}
\frac{\dif f}{f} = -\frac{s^2}{h^2 \bigl(h^2 - s^2\bigr)}\dif\thinspace (h^2).
\end{equation*}
Hence we can summarise our results as follows.
\begin{prop}
\label{5d_geometry}
On the level sets $\nu_M^{-1}(s)$, with $s \neq 0$ regular value for $\nu_M$, we get
\begin{align}
\label{dtheta}
s\dif \vartheta_1 & = \dif \alpha_1 - \frac{1}{h^2 - s^2}\bigl(3g_{UV}\alpha_1 + 3g_{VV}\alpha_2\bigr) \wedge \alpha_0, \\
s\dif \vartheta_2 & = \dif \alpha_2 + \frac{1}{h^2 - s^2}\bigl(3g_{UU}\alpha_1 + 3g_{UV}\alpha_2\bigr) \wedge \alpha_0.
\end{align}
\end{prop}
\begin{prop}
\label{3d_geometry}
Define $f \coloneqq 4h^2/(h^2-s^2)$. The relations among $\alpha_0,\alpha_1,\alpha_2$ on the $T^2$-reduction at level $s \neq 0$ are given by
\begin{equation}
\label{relations_alpha}
\dif \alpha_0 = f\alpha_1 \wedge \alpha_2, \quad
\dif \alpha_1 \wedge \alpha_0 = -\frac{\dif f}{f} \wedge \alpha_1 \wedge \alpha_0, \quad
\dif \alpha_2 \wedge \alpha_0 = -\frac{\dif f}{f} \wedge \alpha_2 \wedge \alpha_0.
\end{equation}
\end{prop}
\begin{remark}
\label{betas}
Define $\beta_0 = \alpha_0$ and $\beta_i = f\alpha_i, i = 1,2$. The equations in \eqref{relations_alpha} are then equivalent to
\begin{equation*}
\dif \beta_0 = \frac{1}{f}\beta_1 \wedge \beta_2, \quad \dif \beta_1 \wedge \beta_0 = 0, \quad \dif \beta_2 \wedge \beta_0 = 0.
\end{equation*}
\end{remark}

\section{Inverse construction}
\label{inverse_process}

Now we wish to invert the construction described so far. Assume we are given a three-dimensional smooth manifold $Q^3$, and let $g_{UU},g_{UV},g_{VV}$ be three $s$-dependent functions on $Q^3$ such that $g_{UU} > 0$ and $g_{UU}g_{VV}-g_{UV}^2>0$. We define the latter quantity as $h^2 \coloneqq g_{UU}g_{VV}-g_{UV}^2$. Let $f > 4$ be a real function and $\alpha_0,\alpha_1,\alpha_2$ be a basis of $s$-dependent one-forms satisfying \eqref{relations_alpha}. Our first goal is to construct a principal $T^2$-bundle over $Q^3$. Let
\begin{align}
\label{big_theta_1}
\Theta_1 & \coloneqq \frac{1}{s}\biggl(\dif \alpha_1 - \frac{3}{h^2-s^2}\bigl(g_{UV}\alpha_1+g_{VV}\alpha_2\bigr) \wedge \alpha_0 \biggr), \\
\label{big_theta_2}
\Theta_2 & \coloneqq \frac{1}{s}\biggl(\dif \alpha_2 + \frac{3}{h^2-s^2}\bigl(g_{UU}\alpha_1+g_{UV}\alpha_2\bigr) \wedge \alpha_0 \biggr),
\end{align}
and find the conditions for which they are closed and with integral period, namely $[\Theta_i] \in H^2(Q_s^3, \mathbb{Z})$. We apply Proposition 2.3, \cite{Twisting}, for this last part. If
\begin{equation*}
\dif \Theta_1 = 0 = \dif{} \biggl(\frac{3}{h^2-s^2}\bigl(g_{UV}\alpha_1+g_{VV}\alpha_2\bigr) \wedge \alpha_0 \biggr)
\end{equation*}
we get
\begin{equation}
\label{1}
\frac{g_{UV}}{h^2}\dif\thinspace (h^2) \wedge \alpha_1 \wedge \alpha_0 + \frac{g_{VV}}{h^2}\dif\thinspace (h^2) \wedge \alpha_2 \wedge \alpha_0 = \dif g_{UV} \wedge \alpha_1 \wedge \alpha_0 + \dif g_{VV} \wedge \alpha_2 \wedge \alpha_0.
\end{equation}
Similarly $\dif \Theta_2 = 0$ yields
\begin{equation}
\label{2}
\frac{g_{UU}}{h^2}\dif \thinspace(h^2) \wedge \alpha_1 \wedge \alpha_0 + \frac{g_{UV}}{h^2}\dif\thinspace (h^2) \wedge \alpha_2 \wedge \alpha_0 = \dif g_{UU} \wedge \alpha_1 \wedge \alpha_0 + \dif g_{UV} \wedge \alpha_2 \wedge \alpha_0.
\end{equation}
Under the conditions \eqref{1} and \eqref{2} one can apply the same process as in Paragraph 2.1, \cite{Twisting}: we construct a principal $T^2$-bundle $E^5 \rightarrow Q^3$ with $s$-dependent connection one-forms $\vartheta_1$ and $\vartheta_2$ such that $\dif \vartheta_k = \Theta_k$, $k=1,2$. The space $E^5$ must be thought of as the level set $\nu_M^{-1}(s)$ of the previous section. So as $s$ varies we have a five-dimensional foliation of a six-dimensional manifold $M = E \times (c,d)$, for some real numbers $c <d$. To construct a nearly K\"ahler structure on $M$ starting from $E^5$, we flow the $\vartheta_k$'s and the $\alpha_i$'s along the normal vector field to $E^5$, given by $\partial / \partial s = \bigl(9(h^2-s^2)\bigr)^{-1}\dif s^{\sharp}$.

In order to establish which equations must be satisfied, we first define $\sigma, \psi_{\pm}$ as in \eqref{sigma}--\eqref{psi_-}, then impose the nearly K\"ahler conditions as we have done above, getting \eqref{theta_2*}--\eqref{forme3*}. Notice that on $M$, the differential $\dif{}$ can be split as the sum of the differential on $E^5$ and the one in the remaining direction: we write $\dif{}_{6} = \dif{}_{5} + \dif{}_{s}$, where $\dif{}_{5}$ is the differential on $E^5$ and $\dif{}_{ s} \gamma = \gamma' \wedge \dif s$, the prime denoting the derivative with respect to $s$ of the form $\gamma$. We use this on \eqref{theta_2*}--\eqref{forme3*} and then contract with $\partial/\partial s$. The equations found will tell us how our forms evolve in the direction defined by $\partial/\partial s$.

We have seen that $\dif\sigma = 3\psi_+$ implies \eqref{theta_1*} in particular. We rewrite this equation as
\begin{align*}
s \dif{}_{ 5}\vartheta_1 + s\dif{}_{ s} \vartheta_1 & = \dif{}_{ 3}\alpha_1 + \dif{}_{ s}\alpha_1 \\
& \qquad - \frac{1}{h^2-s^2}\bigl(3g_{UV}\alpha_1 \wedge \alpha_0 + 3g_{VV}\alpha_2 \wedge \alpha_0 + s \alpha_1 \wedge \dif \nu_M \bigr).
\end{align*}
By assumption, on $E^5$ we have
\begin{equation*}
s \dif{}_{ 5}\vartheta_1 = \dif{}_{ 3}\alpha_1 - \frac{1}{h^2 - s^2}\bigl(3g_{UV}\alpha_1 + 3g_{VV}\alpha_2\bigr) \wedge \alpha_0,
\end{equation*}
so we can simplify our equation getting
\begin{equation*}
\vartheta_1' \wedge \dif s = \frac{1}{s}\alpha_1' \wedge \dif s - \frac{1}{h^2-s^2}\alpha_1 \wedge \dif s.
\end{equation*}
Contracting with $\partial/\partial s$, we obtain
\begin{equation}
\label{result1}
\vartheta_1' = \frac{1}{s}\alpha_1' - \frac{1}{h^2-s^2}\alpha_1.
\end{equation}
Similarly, from \eqref{theta_2*} we have
\begin{equation}
\label{result2}
\vartheta_2' = \frac{1}{s}\alpha_2' - \frac{1}{h^2-s^2}\alpha_2.
\end{equation}
From \eqref{forme1*} and \eqref{alpha_0} we get
\begin{align}
\label{result3}
\vartheta_2' \wedge \alpha_1 - \vartheta_1'\wedge \alpha_2 & = \frac{-4h^2-6s^2+3s(h^2)'}{3\bigl(h^2-s^2\bigr)^2}\alpha_1 \wedge \alpha_2  \nonumber \\
& \qquad + \frac{s}{h^2-s^2}(\alpha_2 \wedge \alpha_1)' + \frac{1}{3\bigl(h^2-s^2\bigr)^2}\dif{}_{ 3}(h^2) \wedge \alpha_0,
\end{align}
and by \eqref{alpha_0} itself we find
\begin{equation}
\label{result4}
\alpha_0' = \frac{4s}{3\bigl(h^2-s^2\bigr)}\alpha_0.
\end{equation}
The remaining equations yield
\begin{align}
\label{result5}
\alpha_0 \wedge \vartheta_2' & = \frac{s}{\bigl(h^2-s^2\bigr)^2}(h^2)'\alpha_2 \wedge \alpha_0 + \frac{1}{3\bigl(h^2-s^2\bigr)^2}\dif{}_{ 3}(h^2) \wedge \bigl(g_{UU}\alpha_1 + g_{UV}\alpha_2\bigr) \nonumber \\
& \qquad + \frac{s}{h^2-s^2}\alpha_2' \wedge \alpha_0 + \frac{h^2-3s^2}{3\bigl(h^2-s^2\bigr)^2}\alpha_2 \wedge \alpha_0 - \frac{1}{3\bigl(h^2-s^2\bigr)}\dif{}_{ 3}\bigl(g_{UU}\alpha_1 + g_{UV}\alpha_2\bigr),
\end{align}
\begin{align}
\label{result6}
\alpha_0 \wedge \vartheta_1' & = \frac{s}{\bigl(h^2-s^2\bigr)^2}(h^2)'\alpha_1 \wedge \alpha_0 - \frac{1}{3\bigl(h^2-s^2\bigr)^2}\dif{}_{ 3}(h^2) \wedge \bigl(g_{UV}\alpha_1 + g_{VV}\alpha_2\bigr) \nonumber \\
& \qquad + \frac{s}{h^2-s^2}\alpha_1' \wedge \alpha_0 + \frac{h^2-3s^2}{3\bigl(h^2-s^2\bigr)^2}\alpha_1 \wedge \alpha_0 + \frac{1}{3\bigl(h^2-s^2\bigr)}\dif{}_{ 3}\bigl(g_{UV}\alpha_1 + g_{VV}\alpha_2\bigr).
\end{align}
If we set $\alpha_1' = \sum_i a_{1i}\alpha_i, \alpha_2' = \sum_j a_{2j}\alpha_j$, and
\begin{equation*}
\dif{}_{ 3}\alpha_1 = \sum_{i<j}b_{ij}\alpha_i \wedge \alpha_j, \quad \dif{}_{ 3}\alpha_2 = \sum_{i < j}c_{ij}\alpha_i \wedge \alpha_j,
\end{equation*}
we can find equations giving $a_{10},a_{11},a_{12},a_{20},a_{21},a_{22},g_{UU}',g_{UV}',g_{VV}', (h^2)'$. Denote by $X_i$ the dual of $\alpha_i$. Using \eqref{result1} and \eqref{result2} in \eqref{result3}, \eqref{result5} and \eqref{result6}, we get
\begin{align*}
(\alpha_2 \wedge \alpha_1)' & = - \frac{s^2}{h^2\bigl(h^2-s^2\bigr)}(h^2)'\alpha_1 \wedge \alpha_2 + \frac{10s}{3\bigl(h^2-s^2\bigr)}\alpha_1 \wedge \alpha_2 - \frac{s}{3h^2\bigl(h^2-s^2\bigr)}\dif{}_{ 3}(h^2) \wedge \alpha_0, \\
 \alpha_0 \wedge \alpha_2' & = - \frac{s^2}{h^2\bigl(h^2-s^2\bigr)}(h^2)'\alpha_0 \wedge \alpha_2 + \frac{2s}{3\bigl(h^2-s^2\bigr)}\alpha_0 \wedge \alpha_2\\
& \qquad + \frac{s}{3h^2\bigl(h^2-s^2\bigr)}\dif{}_{ 3}(h^2) \wedge \bigl(g_{UU}\alpha_1 + g_{UV}\alpha_2\bigr) - \frac{s}{3h^2}\dif{}_{ 3}\bigl(g_{UU}\alpha_1 + g_{UV}\alpha_2\bigr),  \\
 \alpha_0 \wedge \alpha_1' & = - \frac{s^2}{h^2\bigl(h^2-s^2\bigr)}(h^2)' \alpha_0 \wedge \alpha_1 + \frac{2s}{3\bigl(h^2-s^2\bigr)}\alpha_0 \wedge \alpha_1\\
& \qquad - \frac{s}{3h^2\bigl(h^2-s^2\bigr)}\dif{}_{ 3}(h^2) \wedge \bigl(g_{UV}\alpha_1 + g_{VV}\alpha_2\bigr) + \frac{s}{3h^2}\dif{}_{ 3}\bigl(g_{UV}\alpha_1 + g_{VV}\alpha_2\bigr).
\end{align*}
From these equations, comparing the coefficients of $\alpha_0 \wedge \alpha_1, \alpha_0 \wedge \alpha_2$ and $\alpha_1 \wedge \alpha_2$, we find
\begin{align}
\label{a_ij_start}
a_{10} & = \frac{s}{3h^2\bigl(h^2-s^2\bigr)}X_2(h^2), \quad a_{20} = \frac{-s}{3h^2\bigl(h^2-s^2\bigr)}X_1(h^2), \\[5pt]
a_{21} & = \frac{g_{UU}s}{3h^2\bigl(h^2-s^2\bigr)}X_0(h^2) - \frac{s}{3h^2}\bigl(X_0(g_{UU})+g_{UU}b_{01}+g_{UV}c_{01}\bigr), \\[5pt]
a_{12} & = \frac{-g_{VV}s}{3h^2\bigl(h^2-s^2\bigr)}X_0(h^2) + \frac{s}{3h^2}\bigl(X_0(g_{VV})+g_{UV}b_{02}+g_{VV}c_{02}\bigr), \\[5pt]
(h^2)' & = -\frac{2h^2}{s} + \frac{h^2-s^2}{3s}\bigl((b_{01}-c_{02})g_{UV}+c_{01}g_{VV}-b_{02}g_{UU}\bigr),\\[5pt]
a_{11} & = \frac{8s}{3\bigl(h^2-s^2\bigr)}+\frac{s}{3h^2}\bigl(X_0(g_{UV})+b_{02}g_{UU}+c_{02}g_{UV}\bigr)-\frac{sg_{UV}}{3h^2\bigl(h^2-s^2\bigr)}X_0(h^2), \\[5pt]
\label{a_ij_end}
a_{22} & = \frac{8s}{3\bigl(h^2-s^2\bigr)}-\frac{s}{3h^2}\bigl(X_0(g_{UV})+b_{01}g_{UV}+c_{01}g_{VV}\bigr)+\frac{sg_{UV}}{3h^2\bigl(h^2-s^2\bigr)}X_0(h^2).
\end{align}
Further, differentiating \eqref{theta_2*}, \eqref{theta_1*}, and repeating the same process, we get
\begin{align}
\label{gs}
g_{UU}' & = g_{UU}\Biggl(\frac{(h^2)'}{h^2-s^2}+\frac{1}{s}+a_{11}-\frac{2s}{3\bigl(h^2-s^2\bigr)}\Biggr) - \frac{h^2}{3s}c_{01}+g_{UV}a_{21}, \\
g_{UV}' & = g_{UV}\Biggl(\frac{(h^2)'}{h^2-s^2}+\frac{1}{s}+a_{22}-\frac{2s}{3\bigl(h^2-s^2\bigr)}\Biggr) - \frac{h^2}{3s}c_{02}+g_{UU}a_{12} + \frac{sX_0(h^2)}{3\bigl(h^2-s^2\bigr)}, \\
g_{UV}' & = g_{UV}\Biggl(\frac{(h^2)'}{h^2-s^2}+\frac{1}{s}+a_{11}-\frac{2s}{3\bigl(h^2-s^2\bigr)}\Biggr) + \frac{h^2}{3s}b_{01}+g_{VV}a_{21} - \frac{sX_0(h^2)}{3\bigl(h^2-s^2\bigr)}, \\
g_{VV}' & = g_{VV}\Biggl(\frac{(h^2)'}{h^2-s^2}+\frac{1}{s}+a_{22}-\frac{2s}{3\bigl(h^2-s^2\bigr)}\Biggr) + \frac{h^2}{3s}b_{02}+g_{UV}a_{12}.
\end{align}
These results imply that $\alpha_0,\alpha_1,\alpha_2,g_{UU},g_{UV},g_{VV}$ can be found from a system of first order ordinary differential equations, so by Cauchy theorem we find a unique local solution on $E^5 \times (s_0-\varepsilon,s_0+\varepsilon)$, for some $\varepsilon > 0$, where $s_0 \neq 0$ is an initial data. Observe that the value of $s_0$ is specified by $f$ and $h$ through the equation $f = 4h^2/(h^2-s_0^2)$. Finally, \eqref{result1} and \eqref{result2}, together with the expressions of $\alpha_1'$ and $\alpha_2'$ found, give differential equations for $\vartheta_1,\vartheta_2$, to which we can apply the same theorem. Thus we have the final result:

\begin{theorem}
\label{evo_equations}
Let $Q^3$ be a smooth $3$-manifold, $f > 4$ a smooth real function on $Q^3$, and $\{\alpha_i\}_{i=0,1,2}$ a basis of one-forms on $Q^3$ satisfying \eqref{relations_alpha}. Suppose there exists a smooth positive definite $G = \left(\begin{smallmatrix} g_{UU} & g_{UV} \\ g_{UV} & g_{VV} \end{smallmatrix} \right)$ on $Q^3$ such that \eqref{1} and \eqref{2} are fulfilled, and that $s=s_0 =  (1-4/f)^{1/2}h$ is constant. Put $h^2 = \det G$, and define $\Theta_1,\Theta_2$ by \eqref{big_theta_1} and \eqref{big_theta_2} for $s=s_0$.

Then, if $\Theta_k$'s have integral periods, there exist a $T^2$-bundle $E^5 \to Q^3$ with connection one-form $(\vartheta_1,\vartheta_2)$, such that $\dif \vartheta_k = \Theta_k$, and an $\varepsilon > 0$ such that $E^5 \times (s_0-\varepsilon,s_0+\varepsilon)$ has a unique nearly K\"ahler structure of the form \eqref{sigma}--\eqref{psi_-}.
\end{theorem}
\begin{proof}
Let $'$ denote differentiation with respect to $s$ and assume the functions $a_{ij}$'s are those listed in \eqref{a_ij_start}--\eqref{a_ij_end}. Then our forms satisfy the equations
\begin{gather}
\label{final_system1}
\alpha_0' = \frac{4s}{3(h^2-s^2)}\alpha_0, \quad
\alpha_1' = \sum_{i=0}^2 a_{1i}\alpha_i, \quad
\alpha_2' = \sum_{j=0}^2 a_{2j}\alpha_j, \\
\vartheta_1' = \frac{1}{s}\alpha_1' - \frac{1}{h^2-s^2}\alpha_1, \quad
\vartheta_2' = \frac{1}{s}\alpha_2' - \frac{1}{h^2-s^2}\alpha_2.
\end{gather}
For the initial data $s=s_0 = (1-4/f)^{1/2}h$ they have a unique solution, which corresponds to a nearly K\"ahler structure on $E^5 \times (s_0-\varepsilon,s_0+\varepsilon)$, for some $\varepsilon > 0$.
\end{proof}

\section{Invariant structures on the Heisenberg group}
\label{heisenberg_group}
In this section we are going to study the construction described above in the particular case where $Q^3$ is the three-dimensional Heisenberg group $H_3$. Making specific choices of the forms involved and assuming \eqref{alpha_0_ls}--\eqref{alpha_2_ls}, we write the equations in Theorem \ref{evo_equations} and solve them getting explicit solutions. Finally, Proposition \ref{general_heisenberg} proves our solution is general.

Let us consider the Heisenberg group $H_3$, i.e.\ the unipotent Lie group given by the upper triangular real matrices of the form $
\left(
\begin{smallmatrix}
1 & a & b \\
0 & 1 & c \\
0 & 0 & 1
\end{smallmatrix}
\right)
$. Its Lie algebra is generated by
$$
E_0 =
\left(
\begin{smallmatrix}
0 & 1 & 0 \\
0 & 0 & 0 \\
0 & 0 & 0
\end{smallmatrix}
\right),
E_1 =
\left(
\begin{smallmatrix}
0 & 0 & 0 \\
0 & 0 & 1 \\
0 & 0 & 0
\end{smallmatrix}
\right),
E_2 =
\left(
\begin{smallmatrix}
0 & 0 & -1 \\
0 & 0 & 0 \\
0 & 0 & 0
\end{smallmatrix}
\right).
$$ They satisfy the commutation relations $\sbr{E_1,E_2} = -E_0, \sbr{E_0,E_1} = \sbr{E_0,E_2}=0$. If $\sigma_i$ is the dual of $E_i$, then we have
\begin{equation*}
\dif \sigma_0 = \sigma_1 \wedge \sigma_2, \quad \dif \sigma_1 = 0, \quad \dif \sigma_2 = 0.
\end{equation*}
Define $\alpha_k \coloneqq f_k(s)\sigma_k$, and set $g_{UU}(s) = g_{VV}(s) \eqqcolon h(s)$, and $g_{UV}(s_0)=0$ for some initial data $s_0 \neq 0$. With this choice, equations \eqref{1} and \eqref{2} are automatically fulfilled. Then, according to Theorem \ref{evo_equations}, there exists a $T^2$-bundle $E^5 \rightarrow H_3$ with connection one-forms $\vartheta_1,\vartheta_2$ satisfying
\begin{equation*}
s\dif{}_{ 5} \vartheta_1 = - \frac{3h}{h^2-s^2}\alpha_2 \wedge \alpha_0, \quad s\dif{}_{ 5} \vartheta_2 = \frac{3h}{h^2-s^2}\alpha_1 \wedge \alpha_0.
\end{equation*}
Equations $\dif{}_{ 3} \alpha_k = 0$, $k=1,2$, imply that all the coefficients $b_{ij},c_{ij}$ vanish. Furthermore, we have an algebraic relation among the $f_k$'s given by $f_0/f_1f_2 = 4h^2/(h^2-s^2)$. Then we can compute: $a_{10}=a_{20}=a_{12}=a_{21}=0, a_{11} = a_{22} = 8s/3(h^2-s^2)$, and $h' = g_{UU}' = g_{VV}' = -h/s, g_{UV}' = 0$. So the following differential equations for $\alpha_0,\alpha_1,\alpha_2,\vartheta_1,\vartheta_2$ hold:
\begin{equation*}
\alpha_0' = \frac{4s}{3\bigl(h^2-s^2\bigr)}\alpha_0,  \quad \alpha_k' = \frac{8s}{3\bigl(h^2-s^2\bigr)} \alpha_k, \quad
\vartheta_k' = \frac{5}{3\bigl(h^2-s^2\bigr)}\alpha_k, \quad k=1,2.
\end{equation*}
Since $h = g_{UU} > 0$, we obtain the expression $h(s) = \lvert s_0h(s_0)\rvert/\lvert s\rvert \eqqcolon C/s$, and the following differential equations for $f_0,f_1,f_2$:
\begin{equation*}
f_0' = -f_0\Biggl(\frac{4s}{3\bigl(h^2-s^2\bigr)} \Biggr), \quad f_k' = -f_k \Biggl(\frac{8s}{3\bigl(h^2-s^2\bigr)} \Biggr), \quad k = 1,2.
\end{equation*}
Hence one can solve them getting
\begin{equation*}
f_0(s) = f_0(s_0)\left(\frac{C^2-s^4}{C^2-s_0^4}\right)^{1/3}, \quad f_k(s) = f_k(s_0)\left(\frac{C^2-s^4}{C^2-s_0^4}\right)^{2/3}, \quad k =1,2.
\end{equation*}
Let us set $f_0(s_0) = f_1(s_0) = f_2(s_0) = f(s_0)^{-1}$. In this case, having the expressions of $h, f_0,f_1,f_2$, we can write equations \eqref{metric}--\eqref{psi_-} explicitly: for $0 \neq s^2 < \lvert C \rvert$ we obtain
\begin{align}
\label{nk_metric}
g & = \frac{s^2}{9\bigl(
C^2-s^4\bigr)}\dif s^{\otimes 2} + \frac{C}{s}\Bigl(\vartheta_1^{\otimes 2} + \vartheta_2^{\otimes 2}\Bigr) \nonumber \\
& \quad + \frac{s^2\bigl(C^2-s_0^4\bigr)}{16C^4}\Biggl(\biggl(\frac{C^2-s_0^4}{C^2-s^4} \biggr)^{1/3}\sigma_0^{\otimes 2}+ \frac{C}{s}\biggl(\frac{C^2-s^4}{C^2-s_0^4} \biggr)^{1/3}\bigl(\sigma_1^{\otimes 2}+\sigma_2^{\otimes 2}\bigr) \Biggr), \\
\label{nk_sigma}
\sigma & = \frac{s^2}{12C^2}\biggl(\frac{C^2-s_0^4}{C^2-s^4} \biggr)^{2/3}\dif s \wedge \sigma_0 + s\vartheta_1 \wedge \vartheta_2 \nonumber \\
& + \frac{C^2-s_0^4}{4C^2}\biggl(\frac{C^2-s^4}{C^2-s_0^4} \biggr)^{2/3}\bigl(\vartheta_2 \wedge \sigma_1 - \vartheta_1 \wedge \sigma_2 \bigr)-\frac{s^3\bigl(C^2-s_0^4\bigr)}{16C^4}\biggl(\frac{C^2-s^4}{C^2-s_0^4} \biggr)^{1/3}\sigma_1 \wedge \sigma_2, \\
\label{nk_real_volume}
\psi_+ & = \frac{1}{3}\dif s \wedge \Biggl(\vartheta_1 \wedge \vartheta_2 + \frac{s^3}{4C^2}\biggl(\frac{C^2-s_0^4}{C^2-s^4} \biggr)^{1/3} \bigl(\vartheta_1 \wedge \sigma_2 - \vartheta_2 \wedge \sigma_1 \bigr)\Biggr) \nonumber \\
& \quad -\frac{C^2-s_0^4}{16C^4}\Biggl(\frac{s^2}{3}\biggl(\frac{C^2-s^4}{C^2-s_0^4} \biggr)^{1/3}\dif s \wedge \sigma_1 \wedge \sigma_2 +Cs\bigl(\vartheta_1 \wedge \sigma_1 + \vartheta_2 \wedge \sigma_2 \bigr) \wedge \sigma_0\Biggr), \\
\label{nk_imaginary_volume}
\psi_- & = \frac{s}{12C}\biggl(\frac{C^2-s_0^4}{C^2-s^4} \biggr)^{1/3}\dif s \wedge \bigl(\vartheta_1 \wedge \sigma_1 + \vartheta_2 \wedge \sigma_2 \bigr) + \frac{C^2-s_0^4}{4C^2}\biggl(\frac{C^2-s^4}{C^2-s_0^4} \biggr)^{1/3}\vartheta_1 \wedge \vartheta_2 \wedge \sigma_0 \nonumber \\
& \quad + \frac{s^2\bigl( C^2-s_0^4\bigr)}{16C^4} \Biggl(s\bigl(\vartheta_1 \wedge \sigma_2 - \vartheta_2 \wedge \sigma_1 \bigr) -\frac{\bigl(C^2-s_0^4\bigr)^{1/3}\bigl(C^2-s^4\bigr)^{2/3}}{4C^2}\sigma_1 \wedge \sigma_2 \Biggr) \wedge \sigma_0.
\end{align}
\begin{remark}
By the expression of the metric $g$ in \eqref{nk_metric} we can observe that the fiber blows up when $s \rightarrow 0$, whereas the remaining four-dimensional subspace collapses to a point. If $s^2 \rightarrow \lvert C\rvert$, the fiber stabilises, a two-dimensional subspace of the base space collapses to a point, and the rest blows up.
\end{remark}
The following proposition states this is indeed a general solution.
\begin{prop}
\label{general_heisenberg}
The nearly K\"ahler structure in \eqref{nk_metric}--\eqref{nk_imaginary_volume} gives a general left-invariant structure for the case of the Heisenberg group $H_3$.
\end{prop}
\begin{proof}
As seen in Remark \ref{betas}, the geometry of $Q^3$ can be described by three one-forms $\beta_0,\beta_1,\beta_2$ satisfying
\begin{equation*}
\dif \beta_0 = \frac{1}{f}\beta_1 \wedge \beta_2, \quad \dif \beta_1 \wedge \beta_0 = 0, \quad \dif \beta_2 \wedge \beta_0 = 0.
\end{equation*}
Denote by $\tau_0, \tau_1, \tau_2$ a dual basis of the Lie algebra of $H_3$ satisfying $\dif \tau_0 = \tau_1 \wedge \tau_2$ and $\dif \tau_1 = \dif \tau_2 = 0$. In our particular case we can observe that $\dif \beta_0 \in \Span \{\tau_1 \wedge \tau_2 \}$, so $\beta_1 \wedge \beta_2 = c_1\tau_1 \wedge \tau_2$ for some real number $c_1 \neq 0$. This happens only when $\beta_1,\beta_2 \in \Span \{\tau_1,\tau_2\}$, so $\dif \beta_1 = \dif \beta_2 = 0$, and then $\dif \beta_1 \wedge \beta_0 = \dif \beta_2 \wedge \beta_0 = 0$, as we wanted.

Therefore $\beta_0 = \frac{c_1}{f}\tau_0+a\tau_1+b\tau_2$ for some real numbers $a,b$. Now define $\sigma_0 \coloneqq f\beta_0$, and choose $\sigma_1,\sigma_2 \in \Span\{\tau_1,\tau_2\}$ so that $\sigma_1$ is $\tilde{g}$-orthogonal to $\sigma_2$, $\lVert \sigma_1 \rVert_{\tilde{g}} = \lVert \sigma_2 \rVert_{\tilde{g}}$, $\sigma_1 \wedge \sigma_2 = \beta_1 \wedge \beta_2$, and $\sigma_1 \wedge \sigma_2 = c_2\tau_1 \wedge \tau_2$, for some positive constant $c_2$. Define $\tilde{\beta_0} \coloneqq \beta_0, \tilde{\beta_1} \coloneqq \sigma_1$, and $\tilde{\beta_2} \coloneqq \sigma_2$. This new dual frame satisfies
\begin{equation*}
\dif \tilde{\beta_0} = \frac{1}{f}\tilde{\beta_1} \wedge \tilde{\beta_2}, \quad \dif \tilde{\beta_1} = \dif \tilde{\beta_2} = 0, \quad \tilde{g} = h \id,
\end{equation*}
As in the end of Section \ref{three_dimensional_quotients}, we can define three one-forms $\tilde{\alpha}_i, i =0,1,2$ such that $\tilde{\beta}_0 \eqqcolon \tilde{\alpha}_0$ and $\tilde{\beta}_i \eqqcolon f\tilde{\alpha}_i, i = 1,2$. Hence
\begin{align*}
\tilde{\alpha}_0(s_0) & = \tilde{\beta}_0(s_0) = \frac{1}{f(s_0)}\sigma_0 \\
\tilde{\alpha}_k(s_0) & = \frac{1}{f(s_0)}\tilde{\beta}_k(s_0) = \frac{1}{f(s_0)}\sigma_k, \quad k=1,2,
\end{align*}
so if $f_0,f_1,f_2$ are such that $\tilde{\alpha}_k = f_k(s)\sigma_k$, we get $f_0(s_0) = f_1(s_0) = f_2(s_0) = f(s_0)^{-1}$. Thus it is always possible to restrict ourselves to the case studied above.
\end{proof}

\providecommand{\bysame}{\leavevmode\hbox to3em{\hrulefill}\thinspace}
\providecommand{\MR}{\relax\ifhmode\unskip\space\fi MR }
\providecommand{\MRhref}[2]{%
  \href{http://www.ams.org/mathscinet-getitem?mr=#1}{#2}
}
\providecommand{\href}[2]{#2}

\small

(G. Russo) Department of Mathematics and Centre for Quantum Geometry of Moduli Spaces, Aarhus University, Ny Munkegade 118, Bldg 1530, DK-8000 Aarhus C, Denmark

\textit{E\/-mail address}: \texttt{giovanni.russo@math.au.dk}

\smallskip

(A.\ F.\ Swann) Department of Mathematics, Centre for Quantum Geometry of Moduli Spaces, Aarhus University, and Aarhus University Centre for Digitalisation, Big Data and Data Analytics, Ny Munkegade 118, Bldg 1530, DK-8000 Aarhus C, Denmark

\textit{E\/-mail address}: \texttt{swann@math.au.dk}

\end{document}